\documentclass[12pt]{amsart}

\advance\hoffset by -0.5 truecm

\usepackage{amssymb}
\usepackage{amscd}
\usepackage{verbatim}
\usepackage{epsfig}
\usepackage{euscript}
\usepackage[colorlinks=true]{hyperref}
\hypersetup{urlcolor=blue, citecolor=red}
\usepackage{color}

\newtheorem{Theorem}{Theorem}[section]
\newtheorem{Corollary}[Theorem]{Corollary}
\newtheorem{Lemma}[Theorem]{Lemma}

\theoremstyle{definition}
\newtheorem{Definition}[Theorem]{Definition}

\def \p{\partial}

\def \<{\langle}
\def \>{\rangle}

\newcommand{\supp}{\operatorname{supp}}

\begin{document}

\title[Travel time tomography in stationary spacetimes]{Travel time tomography in stationary spacetimes}

\author[Gunther Uhlmann]{Gunther Uhlmann}
\address{Department of Mathematics, University of Washington, Seattle, WA 98195-4350, USA; Jockey Club Institute for Advanced Study, HKUST, Clear Water Bay, Hong Kong, China} 
\email{gunther@math.washington.edu}

\author[Yang Yang]{Yang Yang}
\address{Department of Computational Mathematics, Science and Engineering, Michigan State University, East Lansing, MI 48824, USA}
\email{yangy5@msu.edu}

\author[Hanming Zhou]{Hanming Zhou}
\address{Department of Mathematics, University of California Santa Barbara, Santa Barbara, CA 93106-3080, USA}
\email{hzhou@math.ucsb.edu}


\begin{abstract}
In this paper, we consider the boundary rigidity problem on a cylindrical domain in $\mathbb R^{1+n}$, $n\geq 2$, equipped with a stationary (time-invariant) Lorentzian metric. We show that the time separation function between pairs of points on the boundary of the cylindrical domain determines the stationary spacetime, up to some time-invariant diffeomorphism, assuming that the metric satisfies some a-priori conditions. 
\end{abstract}
\maketitle

%

\section{Introduction and Main Results}

Many inverse problems arise naturally in the study of physical astronomy. This is because modern astronomical observations are typically remotely sensed and need to be transformed into stable and physically meaningful representation of the source, and such transformations are normally accomplished by solving certain inverse problems. As the general theory of relativity proposes Lorentzian geometry as the underlying structure of spacetime, a suitable formulation of these problems is in such a geometric setting. For this reason inverse problems in Lorentzian geometry have been drawing more attention recently, see for instance \cite{ADH96, LOY16, KLU18, LOSU18, HU19}. The goal of these works is to determine some properties of the geometry from various timelike or lightlike observations.

In this article we are interested in determining certain almost-flat Lorentzian metrics from the measurement of time separation functions. We start by recalling some terminologies from Lorentzian geometry in order to formulate the problem. Given a Lorentzian manifold $(M,g)$ of signature $(-1,1,\dots,1)$, a point $p\in M$ is referred to as an event. A tangent vector $\zeta$ is called \emph{timelike}, \emph{null} or \emph{spacelike} if $g(\zeta,\zeta)<0, =0,$ or $>0$, resp.  Given a smooth curve $\gamma$, we say $\gamma$ is \emph{timelike}, \emph{null} or \emph{spacelike} if $\dot{\gamma}(s)$ is timelike, null or spacelike for each $s$, with the addition that $\gamma$ is called \emph{causal} if $\dot{\gamma}(s)$ is either timelike or null for each $s$. 
We say $M$ is \emph{time-orientable} if there exists a global smooth timelike vector field $X$ on $M$. Given such an $X$, we say that a tangent vector $\zeta$ is \emph{future-pointing} if $g(\zeta,X)<0$ and \emph{past-pointing} if $g(\zeta,X)>0$.  
We use the notation $z\ll y$ if there exists a future-pointing timelike curve starting at $z$ and ending at $y$ (i.e., $y$ is in the future of $z$, or alternatively, $z$ is in the past of $y$) and $z\leq y$ if there is a future-pointing causal curve starting at $z$ and ending at $y$. The set $I^+(z):=\{y:z\ll y\}$ is called the chronological future of $z$, and the set $J^+(z):=\{y:z\leq y\}$ is called the causal future of $z$.
\begin{Definition}[\cite{ON90} Chapter 14, Definition 15]
    For a causal curve $\gamma:[a,b]\to M$, let 
    $$L(\gamma):=\int_a^b\sqrt{-g_{\mu\nu}(\gamma(s))\dot{\gamma}^\mu(s)\dot{\gamma}^\nu(s)}ds.$$
    We define the {\it time separation function} $\tau_g$ as
    $$\tau_g(z,y)=\begin{cases}
    	\sup\{L(\gamma):\gamma\text{ is a causal curve from $z$ to $y$}\}&\text{if }y\in J^+(z),\\
    	0&\text{otherwise.}
    \end{cases}$$
\end{Definition}
Given two events $z,y\in M$ with $z\ll y$, there exists a future-pointing timelike curve from $z$ to $y$. After a reparametrization we may assume $\gamma: [0,T]\rightarrow M$ with $g(\dot{\gamma}(s),\dot{\gamma}(s))=-1$ and $\gamma(0)=z$, $\gamma(T)=y$. It is easy to see that $T=L(\gamma)$ and we will call $T$ the {\it proper time} between $z$ and $y$. It has the physical interpretation that this is the elapsed time recorded by a clock which passes through $z$ and $y$ along $\gamma$ in the spacetime. When the supremum is taken on, $\tau_g(z,y)$ is the proper time of the slowest trip in $M$ from $z$ to $y$. If $\tau_g(z,y)>0$, one can check that the maximum is achieved when $\gamma$ is a timelike geodesic from $z$ to $y$. 
The function $\tau_g$ involves time orientation hence is not symmetric except for the trivial case. Some basic properties of time separation functions are studied in \cite[Lemma 5]{LOY16}.


Let $(\mathbb{R}^{1+n},\delta)$, $n\geq 2$, be the standard Minkowski spacetime with $\delta$ the Minkowski metric of the signature $(-1,1,\dots,1)$, i.e. $\delta=-dt^2+e=-dt^2+(dx^1)^2+\cdots+(dx^n)^2$ with $e$ the standard Euclidean metric. Let $g_j$, $j=1,2$ be two {\it standard stationary} Lorentzian metrics on $\mathbb R^{n+1}$, such that 
$$g_j(t,x)=-\lambda_j(x) dt^2+\omega_j(x) \otimes dt+dt\otimes \omega_j(x)+h_j(x), \quad j=1,2.$$
Here $\lambda_j$ is a smooth positive function and $\omega_j$ is a smooth one-form defined on $\mathbb R^n$, $h_j$ is a smooth Riemannian metric on $\mathbb R^n$, so the metric $g_j$ is time invariant. The vector field $\p_t$ defines the global orientation. Stationary spacetimes are important objects in the Einstein field equations and the theory of black holes.


We denote by $\Omega$ a bounded domain in $\mathbb{R}^{n}$ with smooth boundary $\partial\Omega$. We assume that $\p \Omega$ is strictly convex w.r.t. the Euclidean metric, and $g_j$, $j=1,2$, differ from the Minkowski metric $\delta$ only on the cylindrical domain $\mathbb R\times \Omega$ in $\mathbb R^{n+1}$. We measure the time separation function $\tau_{g_j}$ of pairs of points on the boundary $\mathbb R\times \p\Omega$. In the meantime, since $g_j$ is time invariant, it suffices to restrict the measurement on $\Gamma:=[0,T]\times \p \Omega$ for some $T>0$ sufficiently large (depending on the metric $g_j$ and the domain $\Omega$). 
The question we would like to consider is the following: does $\tau_{g_1}|_{\Gamma\times \Gamma} = \tau_{g_2}|_{\Gamma\times \Gamma}$ imply $g_1=g_2$? Notice that $g_1=g_2=\delta$ outside $\mathbb R\times \Omega$.

The answer is no in general. An obvious obstruction exists: let $\psi:\overline\Omega\to \overline\Omega$ be a diffeomorphism with $\psi|_{\p\Omega}=Id$, and define the diffeomorphism $\Psi:\mathbb R\times\Omega\to \mathbb R\times \Omega$ by $\Psi(t,x):=(t,\psi(x))$, i.e. $\Psi=Id\times \psi$. Then it is easy to check that $\tau_{\Psi^\ast g}|_{\Gamma \times \Gamma}=\tau_g|_{\Gamma \times \Gamma}$. 
This suggests that the above question should be considered modulo such diffeomorphisms.


To obtain an affirmative answer to the above question, we will impose further restrictions on the metrics, namely $g_j$ is close to $\delta$. To be more precise, let $C^k_0(\overline{\Omega})$ denote the space of $k$-differentiable tensors $f$ with $\partial^\alpha f=0$ on $\partial\Omega$ for $|\alpha|\leq k$, we require that for $j=1,2$,
\begin{equation} \label{metricclose}
\|\lambda_j-1\|_{C^k_0(\overline{\Omega})} < \epsilon,\quad \|\omega_j\|_{C^k_0(\overline{\Omega})} < \epsilon,\quad \|h_j-e\|_{C^k_0(\overline{\Omega})} < \epsilon
\end{equation}
for some small $\epsilon>0$ and some index $k$ to be specified later. It's useful to mention that the convexity of $\p \Omega$ is preserved under small perturbation of the metric. Moreover, we impose the following compatible condition for $\omega_j$ and $h_j$.

\begin{Definition}
Given a standard stationary Lorentzian metric 
$$g=-\lambda dt^2+\omega\otimes dt+dt\otimes \omega+h$$
on $\mathbb R^{1+n}$, which differs from the Minkowski metric $\delta$ only on $\mathbb R\times \Omega$. Suppose $g$ is close to $\delta$ in the sense of \eqref{metricclose} for $\epsilon>0$ sufficiently small. We say that $g$ satisfies the {\it orthogonal assumption} if there exists a hyperplane $H\subset \mathbb R^n\setminus \Omega$, so that $\omega(\dot\sigma)=0$ along any geodesic $\sigma$, w.r.t. the Riemannian metric $h$, normal to $H$.

Given two such metrics $g_1$ and $g_2$ satisfying the orthogonal assumption w.r.t. the same hyperplane $H$. We say that $g_1$ and $g_2$ satisfy the {\it spatial distance assumption} w.r.t. the hyperplane $H$ if $d_{h_1}(x,y)=d_{h_2}(x,y)$ for all $x,y\in \p \Omega$ with $x-y$ almost normal (in Euclidean product) to $H$, i.e. the angle between $x-y$ and $H$ is in some small neighborhood of $\pi/2$. Here $d_h$ is the Riemannian distance function w.r.t. metric $h$.
\end{Definition}
The orthogonal assumption implies that there exist global coordinates $(x^1,\cdots, x^n)$, in which $\omega$ can be written as $\omega=\omega_2 dx^2+\cdots+\omega_n dx^n$, i.e. $\omega_1\equiv 0$. The spatial distance assumption also appears in \cite{Wa99} on the stability estimate of the Riemannian boundary rigidity problem. On the other hand, the spatial dimension is $\geq 2$, the spatial distance assumption only requires that $d_{h_1}=d_{h_2}$ on a small subset of $\p\Omega\times \p\Omega$. See Section \ref{construct diffeomorphism} for the details. 


Next we define sets of pairs of stationary metrics satisfying some a-priori estimates
$$\mathcal G_K:=\{(g_1,g_2)\; : \; \|g_1-g_2\|_{H^2(\overline\Omega)}\leq K \|g_1-g_2\|_{H^1(\overline\Omega)}\},\quad \quad K\geq 1.$$
We have the following main result of this article.


\begin{Theorem}\label{main theorem}
Let $g_1$ and $g_2$ be two standard stationary Lorentzian metrics in $\mathbb{R}^{1+n}$, $n\geq 2$, which differ from the standard Minkowski metric $\delta$ only on $\overline{\Omega}$, and $(g_1,g_2)\in \mathcal G_K$ for some constant $K\geq 1$. Suppose these metrics are close to $\delta$ in the $C^{k}_0(\overline{\Omega})$-norm, $k\geq 2n+6$, in the sense of \eqref{metricclose} for some sufficiently small $\epsilon>0$ (depending on $K$), and satisfy the orthogonal and spatial distance assumption w.r.t. the same hyperplane $H$. 
There exists $T_0>0$ depending on $\Omega$ and $\epsilon$, for any $T>T_0$, if $\Gamma=[0,T]\times \p\Omega$ and
$$\tau_{g_1}|_{\Gamma\times \Gamma} = \tau_{g_2}|_{\Gamma\times \Gamma},$$
then there exists a diffeomorphism $\psi$ on $\overline \Omega$ with $\psi|_{\p\Omega}=Id$, such that for $\Psi=Id\times \psi$, $g_2=\Psi^\ast g_1$.
\end{Theorem}

Note that $g_2=\Psi^\ast g_1$ is equivalent to the statement that $\lambda_2=\psi^\ast \lambda_1$, $\omega_2=\psi^\ast \omega_1$ and $h_2=\psi^\ast h_1$.



There are very few rigidity results in the Lorentzian context. It is known that the time separation function determines the Lorentzian metrics, up to diffeomorphism, on flat two dimensional product Lorentzian manifolds \cite{ADH96} and on universal covering spaces of real-analytic Lorentzian manifolds \cite{LOY16} (under additional assumptions). Our result works for smooth Lorentzian metrics in dimensions $3$ and higher.

An analogous inverse problem in Riemannian geometry exists and has been extensively studied. The problem is typically known as the {\it boundary rigidity problem}, and it asks to what extent can one determine a Riemannian metric on a compact smooth manifold with boundary from the distance between boundary points. In geophysical literature, this problem is known as the {\it travel time tomography}, which is concerned with the recovery of the inner structure of the Earth from the travel times of seismic waves at the surface \cite{He1905, WZ1907}. Rigidity results have been established when the metrics are close to the Euclidean one \cite{SU98, LSU03, BI10}. For general geometry, Michel conjectured that simple manifolds are boundary rigid \cite{Mi81}. We recall that a compact Riemannian manifold with boundary is {\it simple} if the boundary is strictly convex and any two points can be joined by a unique distance minimizing geodesic. Boundary rigidity is shown on simple Riemannian surfaces \cite{PU05} and generic simple metrics in dimensions $\geq 3$ \cite{SU05}, including the real-analytic ones. We refer to the recent survey \cite{SUVZ19} and the references therein for the developments in the Riemannian case.


Theorem \ref{main theorem} is a generalization of the Riemannian case \cite{SU98} to the Lorentzian geometry. When the Lorentzian metric $g=-dt^2+h$ on $\mathbb R\times \Omega$, where $h$ is a Riemannian metric, then it is easy to check that the spatial projection of an arbitrary time-like geodesic is a Riemannian geodesic w.r.t. the metric $h$. In particular, one can determine the time separation function $\tau_g$ from the length of Riemannian geodesics connecting the boundary $\p\Omega$ w.r.t. $h$. Now a weaker version of the main result of \cite{SU98} follows immediately from Theorem \ref{main theorem}. Notice that \cite{SU98} considers only the $3$ dimensional case, we improve the method so that it works in any dimension $\geq 2$. Let $d_h: \p\Omega\times \p\Omega\to [0,\infty)$ be the Riemannian boundary distance function.

\begin{Corollary}
Let $h_1$ and $h_2$ be two Riemannian metrics in $\mathbb{R}^{n}$, $n\geq 2$, which differ from the Euclidean metric $e$ only on $\overline{\Omega}$, and $(h_1,h_2)\in \mathcal G_K$ for some constant $K\geq 1$. Suppose $\|h_j-e\|_{C^{k}_0(\overline{\Omega})}\leq \epsilon$, $j=1,2$, $k\geq 2n+6$, for some sufficiently small $\epsilon>0$ (depending on $K$), and
$$d_{h_1}|_{\p\Omega \times \p\Omega} = d_{h_2}|_{\p\Omega \times \p\Omega},$$
then there exists a diffeomorphism $\psi$ on $\overline \Omega$ with $\psi|_{\p\Omega}=Id$, such that $h_2=\psi^\ast h_1$.
\end{Corollary}

The assumption that $(h_1,h_2)\in \mathcal G_K$ for some $K\geq 1$ is not needed in \cite{SU98}. Indeed, for the Riemannian case, by combining the ideas of \cite{SU98} and the current article, we can drop this assumption, and generalize the main theorem of \cite{SU98} to arbitrary dimension $\geq 2$. 

\begin{Theorem}\label{Riemannian case}
Let $h_1$ and $h_2$ be two Riemannian metrics in $\mathbb{R}^{n}$, $n\geq 2$, which differ from the Euclidean metric $e$ only on $\overline{\Omega}$. Suppose $\|h_j-e\|_{C^{k}_0(\overline{\Omega})}\leq \epsilon$, $j=1,2$, $k\geq 2n+6$, for some sufficiently small $\epsilon>0$, and
$$d_{h_1}|_{\p\Omega \times \p\Omega} = d_{h_2}|_{\p\Omega \times \p\Omega},$$
then there exists a diffeomorphism $\psi$ on $\overline \Omega$ with $\psi|_{\p\Omega}=Id$, such that $h_2=\psi^\ast h_1$.
\end{Theorem}


To prove the main theorem, we adopt an approach similar to the one of \cite{SU98}. We first show that the time separation function determines the scattering relation of the Hamiltonian flow of time-like geodesics. Then a key ingredient of the method is an integral identity derived in \cite{SU98}. In our case, it implies that the difference of the scattering relations of $g_1$ and $g_2$ (in coordinates) equals some weighted ray transform of $g_1-g_2$ and $\nabla (g_1-g_2)$ along time-like geodesics. The integral identity also plays an important role in recent advances on boundary and lens rigidity problems for Riemannian metrics \cite{SUV16, SUV17}, and other rigidity problems \cite{LSSU17, Zh18, PUZ19}.
The study of the injectivity of this ray transform takes the main part of article. We analyze the transform as a Fourier integral operator (FIO), and apply a perturbation argument (note that the metrics are close to the Minkowski metric) to achieve the invertibility. We remark that there are previous studies of ray transforms on Lorentzian manifolds \cite{St89, St17, LOSU18, Wa18, LOSU19, FIKO19, FIO20}, which all consider the case of null geodesics (i.e. light ray transform). In the current paper, we focus on time-like geodesics. Previous studies of weighted (Riemannian) geodesic ray transforms can be found in e.g. \cite{FSU08, PSUZ19, Zh17}.

The paper is structured as follows. In Section 2 we establish some preliminary results to be used in the proofs later. Section 3 is devoted to the proof of Theorem \ref{main theorem}. A sketch of the proof of Theorem \ref{Riemannian case} is given in Section 4.

\bigskip

\noindent {\bf Acknowledgment:} GU was partly supported by NSF, a Walker Family Endowed Professorship at UW and a Si-Yuan Professorship at HKUST. YY was partly supported by NSF grant DMS-1715178, DMS-2006881, and start-up fund from MSU.

\section{Preliminaries}

\subsection{Recovering the scattering relation from the time separation function}

We first prove a lemma to be used in the establishment of the main integral identity. The lemma roughly says that two timelike geodesics with respect to $g_1$, $g_2$, respectively which enter $\mathbb R\times \Omega$ at the same point in the same direction will exit from the same point in the same direction, provided that $\tau_{g_1}|_{\Gamma\times\Gamma} = \tau_{g_2}|_{\Gamma\times\Gamma}$.

Given a Lorentzian manifold $(M,g)$, we define the associated Hamiltonian $H_g := \frac{1}{2}\sum_{j,k=0}^{n} g^{jk}\zeta_j\zeta_k$
where $g^{-1}=(g^{jk})$ is the inverse of $g=(g_{jk})$, and the Hamiltonian vector field
$$V_g:=\left(\frac{\partial H_g}{\partial\zeta}, -\frac{\partial H_g}{\partial z}\right).$$
Fix a point $z_{(0)}\in M$ and a covector $\zeta^{(0)}\in T^\ast_{z_{(0)}}M$, 
we denote by $(z_g(s),\zeta_g(s))$ the bi-characteristic curve initiated from $(z_{(0)},\zeta^{(0)})$, that is, the solution of the Hamiltonian system, in coordinates $z=(z^0, z^1,\cdots, z^n)$,
\begin{equation} \label{Hamsys1}
\left\{
\begin{array}{rclrcl}
\displaystyle\frac{d z^j_g}{ds} & = & \displaystyle\sum^n_{l=0} g^{jl}\zeta_l, &\quad\quad z_g(0) &=& z_{(0)} \\
\displaystyle\frac{d \zeta_{gj}}{ds} & = & \displaystyle -\frac{1}{2}\sum^n_{l, m=0} \frac{\partial g^{lm}}{\partial z^j}\zeta_l\zeta_m, & \quad\quad \zeta_g(0) &=& \zeta^{(0)}. 
\end{array}
\right.
\end{equation}

Let us make two definitions on covectors. We say that a covector $\zeta$ is {\it timelike} if its dual vector $g^{-1}\zeta=(g^{ij}\zeta_j)$ is timelike. In the meantime, $\zeta$ is {\it future (past) pointing} if its dual vector $g^{-1}\zeta$ is future (past) pointing.

\begin{Lemma} \label{srelation}
Let $g_j$ ($j=1,2$) be two Lorentzian metrics in $\mathbb{R}^{1+n}$ which differ from $\delta$ only in $\mathbb R\times \overline{\Omega}$. Let $z_{(0)}\in \{0\}\times \p\Omega$, $\zeta^{(0)}\in T^*_{z_{(0)}}\mathbb R^{1+n}$ be a future-pointing timelike covector. Denote by $(z_{g_j}, \zeta_{g_j})$ the solution to \eqref{Hamsys1} of initial value $(z_{(0)}, \zeta^{(0)})$, with $g$ replaced by $g_j$, $j=1,2$. Let $\ell>0$ be the first time such that $z_{g_1}(\ell)\in \Gamma=[0,T]\times\p\Omega$ for some $T>0$ sufficiently large. 
If $\tau_{g_1}|_{\Gamma\times \Gamma} = \tau_{g_2}|_{\Gamma \times \Gamma}$, then
$$z_{g_1}(\ell)=z_{g_2}(\ell)\in \Gamma,\quad\zeta_{g_1}(\ell)=\zeta_{g_2}(\ell).$$
\end{Lemma}

\noindent {\bf Remark:} On $\mathbb R^{1+n}$, we have the global coordinates $z=(z^0, z^1,\cdots, z^n)$ with $z^0=t$, $(z^1,\cdots, z^n)=x$. 
Notice that geodesics have constant speed, i.e. $|\zeta |_{g}=const$, we do not require the curve $z_{g_j}$ to have unit speed, i.e. the time $\ell>0$ does not need to be the proper time. 
 
\begin{proof}
Consider the gradient $\nabla_z \tau_{g_j}(z_{(0)},z)$ for $z=(t,x)\in I^+(z_{(0)})$, i.e. $z$ is in the future of $z_{(0)}$. Then by the definition of the time separation function, one can check that
\begin{equation}\label{eikonal}
g_j(\nabla_z \tau_{g_j}(z_{(0)},z),\nabla_z \tau_{g_j}(z_{(0)},z))=-1.
\end{equation}
By the assumption $\tau_{g_1}(z_{(0)},z)=\tau_{g_2}(z_{(0)},z)$, we get that
$$\p_t \tau_{g_1}(z_{(0)},z)=\p_t \tau_{g_2}(z_{(0)},z), \quad \nabla_T \tau_{g_1}(z_{(0)},z)=\nabla_T \tau_{g_2}(z_{(0)},z),$$
where $\nabla_T$ is the tangential gradient of the spacial variables $x$ with respect to the boundary $\p \Omega$.
Notice 
$$\nabla_z \tau(z_{(0)},(t,x))=(\p_t \tau(z_{(0)},(t,x)), \nabla_T \tau(z_{(0)},(t,x)), \p_\nu \tau(z_{(0)},(t,x)),$$
where $\p_\nu$ is the spacial normal derivative with respect to $\p\Omega$. Since $g_1|_\Gamma=g_2|_\Gamma=\delta$, \eqref{eikonal} implies that $\p_\nu \tau_{g_1}(z_{(0)},z)=\p_\nu \tau_{g_2}(z_{(0)},z)$. Therefore
$$\nabla_z \tau_{g_1}(z_{(0)},z_{g_1}(\ell))=\nabla_z \tau_{g_2}(z_{(0)},z_{g_1}(\ell)),$$ 
which implies that the time-like geodesics $z_{g_1}$ and $y_{g_2}$ w.r.t. $g_1$ and $g_2$ respectively, connecting points $z_{(0)}$ and $z_{g_1}(\ell)$ must be tangent at the end point $z_{g_1}(\ell)$. Similarly, one can show that they must be tangent at the initial point $z_{(0)}$ as well. 

We conclude that the $g_2$ time-like geodesics $z_{g_2}$ and $y_{g_2}$ share the same initial data, by the uniqueness of ODEs, this implies that $z_{g_2}\equiv y_{g_2}$. In particular, since $g_1=g_2$ on $\Gamma$ and $\tau_{g_1}(z_{(0)}, z_{g_1}(\ell))=\tau_{g_2}(z_{(0)}, z_{g_1}(\ell))$, we get $z_{g_1}(\ell)=z_{g_2}(\ell)$ and $\zeta_{g_1}(\ell)=\zeta_{g_2}(\ell)$.
\end{proof}

The Riemannian case of the above lemma was established in \cite{Mi81}.

\noindent {\bf Remark:} {\it A stationary metric $g$ is invariant in time, therefore even though we only know the information of $\tau_g(z,y)$ for $z,y\in \Gamma$, we have that $\tau_g(z,y)=\tau_g(\Lambda_s z,\Lambda_s y)$, $\forall s\in\mathbb R$, where $\Lambda_s$ is the translation by $s$ in the time direction, i.e. $\Lambda_s (t,x)=(t+s,x)$. Therefore, it's enough to consider $z_{(0)}\in \{0\}\times \p\Omega$ in Lemma \ref{srelation}. On the other hand, in the proof of the main theorem, we only need to take use of those time-like geodesics which are close to light-like ones, i.e. whose tangent vectors are close to light-like ones. The above analysis implies that it suffices to measure the time separation function on $\Gamma=[0,T]\times \p \Omega$ for any $T>T_0$, where $T_0$, roughly speaking, is approximately equal to the (Euclidean) diameter of $\Omega$.}

\subsection{Pullback of the metrics by diffeomorphisms}\label{construct diffeomorphism}

In this part we simplify the form of the metrics $g_1$ and $g_2$. We will see that to solve the original problem, it is enough to consider a similar problem where the two metrics are of special structures and satisfy the same conditions as $g_1$ and $g_2$. In the rest of the paper we name various positive constants which are independent of $\epsilon$ as $C$ in the estimates.

Since $g_j$, $j=1,2$, satisfy the orthogonal assumption for the same hyperplane, there exists a hyperplane $H\subset \mathbb R^n\setminus \Omega$ and global coordinates $(x^1,x^2,\cdots,x^n)$ so that the $H=\{x^1=0\}$. 
Recall that 
$$g_j=-\lambda_j dt^2+\omega_j\otimes dt+dt\otimes \omega_j+h_j,$$
the orthogonal assumption implies that $\omega_j(\dot\sigma_j)=0$ along any $h_j$-geodesic $\sigma_j$ normal to $H$.

The assumption that $h_1$ and $h_2$ are close to the Euclidean metric implies their bi-characteristics are close to the straight lines in $\Omega\subset \mathbb R^n$. 
More precisely, we work in the global coordinates $(x^1,\cdots,x^n)$ induced by the hyperplane $H$, for $x_0\in \mathbb R^{n-1}$, denote by $x=x(s,x_0)$ and $\xi=\xi(s,x_0)$ the solution to the Hamiltonian system with respect to the $h$-geodesic flow
\begin{equation} \label{Hamsys}
\left\{
\begin{array}{rclrcl}
\displaystyle\frac{d x^j}{ds} & = & \displaystyle\sum^n_{l=1} h^{jl}\xi_l, & x(0) &=& (0,x_0) \\
\displaystyle\frac{d \xi_j}{ds} & = & \displaystyle -\frac{1}{2} \sum^n_{l, m=1} \frac{\partial h^{lm}}{\partial x^j}\xi_l\xi_m, & \xi(0) &=& (1, 0, \dots, 0) \in \mathbb{R}^{n}. 
\end{array}
\right.
\end{equation}
where $h$ is either $h_1$ or $h_2$. The hypothesis \eqref{metricclose} combined with elementary ODE theory claims 
\begin{equation} \label{solclose}
\|x-(s,x_0)\|_{C^{k-1}} + \|\xi - (1,0, \dots, 0)\|_{C^{k-1}} \leq C\epsilon
\end{equation}
for some constant $C>0$ which is uniform in $(s,x_0)$ on compact sets. Notice that $x(0)\in H\subset \mathbb R^n\setminus \Omega$, so $h(x(0))=e$, which implies that the solution of \eqref{Hamsys} are geodesics normal to $H$.

We make a change of variable to ``straighten'' the geodesics of $h$. Let $x=x(s,x_0)$ be the solution to the system \eqref{Hamsys} and let $y=y(s):=(s,x_0)$. The latter is actually the solution to the system \eqref{Hamsys} when $h=e$. Define the map $\psi(y):=x$. This map is close to the identity map in the $C^{k-1}(\overline{\Omega})$ topology due to \eqref{solclose}, it is therefore a diffeomorphism. Moreover, the pull-back metric $\psi^\ast h$ in the coordinates $y$ takes the form
\begin{equation} \label{specialform}
\left(
\begin{array}{cc}
	1 & 0_{1\times (n-1)} \\
	0_{(n-1)\times 1} & \ast_{(n-1)\times (n-1)}
\end{array}
\right)
\end{equation}
where $\ast$ denotes an $(n-1)\times (n-1)$ matrix-valued function of $y$. 

Then we consider the pull-back one form $\psi^\ast \omega$. By the orthogonal assumption, it is easy to see that $h^{-1}(\omega,\xi)=0$, where $\xi$ is the solution to the Hamiltonian system \eqref{Hamsys}. Notice that $\psi^\ast \xi=(1,0,\cdots,0)$, this implies the pull-back one form $\psi^\ast \omega$ in the $y$ coordinates takes the form
\begin{equation}\label{specialform 2}
( 0, \,\,\ast_{1\times (n-1)}),
\end{equation}
where $\ast$ is a $1\times (n-1)$ vector-valued function of $y$.


Denote by $\psi_1$, $\psi_2$ the change of variables $y\mapsto x$ related to $h_1$, $h_2$, respectively. By applying the idea of the proof of Lemma \ref{srelation} to the Riemannian case (see also \cite[Lemma 2.1]{SU98}), the spatial distance assumption for $g_1$ and $g_2$ implies that if $x, y\in \p\Omega$ and $x-y$ is almost normal to $H$, then two geodesics connecting $x$ and $y$ w.r.t. $h_1$ and $h_2$ must be tangent with each other at $x$ and $y$. In particular, $h_1$ and $h_2$ share the same scattering relation for $(x,\xi)\in T^*\mathbb R^n$, $x\in \p\Omega$ and $\xi=(1,0,\cdots,0)$.  
Then one can deduce that $\psi_1=\psi_2$ in $\mathbb{R}^{n}\backslash \Omega$. 
so $\psi_1$ and $\psi_2$ map $\Omega$ to the same new domain, say $\tilde{\Omega}$. Denote $\Psi_j:=Id\times \psi_j:\mathbb R^{1+n}\to \mathbb R^{1+n}$, $j=1,2$, and let $\tilde{g}_1:=\Psi^\ast_1 g_1$ and $\tilde{g}_2:=\Psi^\ast_2 g_2$, it follows
$$\|\tilde{g}_j - \delta \|_{C^{k-2}} \leq C\epsilon, \quad j=1,2$$
for some constant $C>0$. By \eqref{specialform} and \eqref{specialform 2}, $\tilde g_j$, $j=1,2$, have the form
\begin{equation} \label{special form of g}
\left(
\begin{array}{ccc}
   \ast & 0 & \ast_{1\times (n-1)}\\[.5em]
	0 & 1 & 0_{1\times (n-1)} \\[.5em]
	\ast_{(n-1)\times 1} & 0_{(n-1)\times 1} & \ast_{(n-1)\times (n-1)}
\end{array}
\right).
\end{equation}
Then it's easy to check that the inverse $\tilde g_j^{-1}$ takes the same form. Moreover, the time separation function is unchanged under the diffeomorphisms, i.e. $\tau_{\tilde g_j}=\tau_{g_j}$, $j=1,2$.

Now we only need to prove $\tilde{g}_1=\tilde{g_2}$, then $\Psi^{-1}_2\Psi_1$ is a (time-invariant) diffeomorphism which fixes $\mathbb R\times \partial \Omega$ and satisfies $(\Psi^{-1}_2\Psi_1)^\ast g_1 =g_2$. 
From now on we assume such a modification has been made so that we can consider the same problem under the additional assumption that the two metrics $g_1$ and $g_2$ take the form \eqref{special form of g}.



\subsection{An integral identity}

The main ingredient of our proof is an integral identity due to Stefanov and Uhlmann \cite{SU98}. We include its derivation here for the sake of completeness and to fix some notations for later discussion.

Notice that the solution of \eqref{Hamsys1} depends on the initial conditions, that is, $z=z(s,z_{(0)},\zeta^{(0)})$, $\zeta=\zeta(s,z_{(0)},\zeta^{(0)})$. We denote $X=(z,\zeta)$, therefore $X_{g_j}=X_{g_j}(s,X_{(0)})$, where $X_{(0)}:=(z_{(0)},\zeta^{(0)})$. Define a function
$$F(s):=X_{g_2}(\ell-s, X_{g_1}(s,X_{(0)}))$$
where $\ell$ is the common time at which the two geodesics $z_{g_j}$ exit $\mathbb R\times \Omega$, see Lemma \ref{srelation}. Note that
$$F(0)=X_{g_2}(\ell,X_{g_1}(0,X_{(0)})) = X_{g_2}(\ell,X_{(0)}) = X_{g_1}(\ell ,X_{(0)})$$
and 
$$F(\ell) = X_{g_2}(0,X_{g_1}(\ell,X_{(0)})) = X_{g_1}(\ell,X_{(0)}).$$
Thus
\begin{equation} \label{intzero}
\int^\ell_0 F'(s) \;ds = 0.
\end{equation}

Let $V_{g_j}$ denote the Hamiltonian vector field associated with $g_j$ for $j=1,2$. Then it is easily seen that
\begin{equation} \label{Fprime}
F'(s) = -V_{g_2}(X_{g_2}(\ell-s, X_{g_1}(s,X_{(0)}))) + \frac{\partial X_{g_2}}{\partial X_{(0)}}(\ell-s, X_{g_1}(s,X_{(0)})) V_{g_1}(X_{g_1}(s,X_{(0)}))
\end{equation}
where $\frac{\partial X_{g_2}}{\partial X_{(0)}}$ denotes the matrix derivative with respect to the initial condition $X_{(0)}$. To rewrite the first term on the right, we observe that $X_{g_2}(\ell-s, X_{g_2}(s,X_{(0)}))$ is independent of $s$, hence
$$0=\frac{d}{ds}\big|_{s=0} X_{g_2}(\ell-s, X_{g_2}(s, X_{(0)})) = -V_{g_2}(X_{g_2}(\ell-s, X_{(0)})) + \frac{\partial X_{g_2}}{\partial X_{(0)}}(\ell-s, X_{(0)})V_{g_2}(X_{(0)}).$$
Substituting this into \eqref{Fprime} we have
$$F'(s)=\frac{\partial X_{g_2}}{\partial X_{(0)}}(\ell-s, X_{g_1}(s, X_{(0)}))(V_{g_1} - V_{g_2})(X_{g_1}(s,X_{(0)})).$$
This combined with \eqref{intzero} yields the desired integral identity
\begin{equation}\label{integral identity}
\int^\ell_0 \frac{\partial X_{g_2}}{\partial X_{(0)}}(\ell-s, X_{g_1}(s, X_{(0)}))(V_{g_1} - V_{g_2})(X_{g_1}(s,X_{(0)})) \;ds=0.
\end{equation}

Recall the Hamiltonian system \eqref{Hamsys1} for the geodesic flows of the Lorentzian metric $g=g_1$ or $g_2$ with initial data $(z_{(0)}, \zeta^{(0)})$.
Let $(z_{(0)}, \zeta^{(0)})$ be such that $z_{(0)}\in \mathbb R^{1+n}\setminus \mathbb R\times \Omega$, $\zeta^{(0)}$ is a timelike covector with respect to $\delta$ and close to light-like. 
Since $g_1=g_2=\delta$ in $\mathbb R^{1+n}\setminus \mathbb R\times \Omega$, $\zeta^{(0)}$ is timelike with respect to $g_1$, $g_2$ as well. 

Denote by $z_g(s,z_{(0)},\zeta^{(0)})$ and $\zeta_g(s,z_{(0)},\zeta^{(0)})$ the solutions of the Hamiltonian  system. When $g=\delta$ the solution is
\begin{equation*} \label{Mbi}
z_\delta = z_{(0)} + s\delta \zeta^{(0)}, \quad\quad \zeta_\delta=\zeta^{(0)}.
\end{equation*}
Here $\delta \zeta=(-\zeta_0,\zeta_1,\cdots, \zeta_n)$ for $\zeta=(\zeta_0,\zeta_1,\cdots,\zeta_n)$.
The closeness condition \eqref{metricclose} implies
\begin{equation} \label{xiclose}
z_g = z_{(0)} + s\delta \zeta^{(0)} + O(\epsilon), \quad\quad \zeta_g = \zeta^{(0)}+O(\epsilon) \quad\quad \text{ in } C^{k-3}.
\end{equation}
Here the $C^{k-3}$ norm is with respect to the variables $s,z_{(0)},\zeta^{(0)}$ 
and $O(\epsilon)$ means functions with norm bounded by $C\epsilon$ with a constant $C>0$ uniform in any fixed compact set. The map $z_\delta\mapsto z_g(s,z_{(0)},\zeta^{(0)})$ is close to the identity map and thus is a diffeomorphism for small $\epsilon$.

For $X_{(0)}=(z_{(0)},\zeta^{(0)})$, taking derivatives in \eqref{xiclose} shows that
$$
\frac{\partial X_{g_2}}{\partial X_{(0)}}(s, z_{(0)}, \zeta^{(0)}) = \left(
\begin{array}{cc}
	1 & s \\
	0 & 1 
\end{array} \right)
 + O(\epsilon). \quad\quad \text{ in } C^{k-4}.
$$
This implies
\begin{equation} \label{derivative}
\frac{\partial X_{g_2}}{\partial X_{(0)}}(\ell-s, X_{g_1}(s,X_{(0)})) = \left(
\begin{array}{cc}
	1 & \ell-s \\
	0 & 1 
\end{array} \right)
 + B(s). \quad\quad \text{ in } C^{k-4}
\end{equation}
where $\ell$ is the common exit time and  
$$B(s):=B(s,X_{(0)}; g_1,g_2) = 
\left(\begin{array}{cc}
	B_{11} & B_{12} \\
	B_{21} & B_{22} 
\end{array} \right) = O(\epsilon) \quad\quad \text{ in } C^{k-4}
$$
is a $2(1+n)\times 2(1+n)$ matrix-valued function and each block $B_{ij}$ is a $(1+n)\times (1+n)$ matrix-valued function. 
Insert \eqref{derivative} into the integral identity \eqref{integral identity} to derive
\begin{equation} \label{identity2}
\int^\ell_0 \left( 
\left(
\begin{array}{cc}
	1 & \ell-s \\
	0 & 1
\end{array}
\right) + B(s)
 \right)
(V_{g_1} - V_{g_2})(X_{g_1}(s,X^{(0)})) \,ds = 0.
\end{equation}

Write $m:=g^{-1}_{1}-g^{-1}_{2}$, $\zeta=\zeta_{g_1}(s,z_{(0)},\xi^{(0)})$, and let $x$ be the spacial part of $z_{g_1}(s,z_{(0)},\zeta^{(0)})$. Simple calculation shows that
\begin{equation}\label{Hfield}
(V_{g_1} - V_{g_2})(X_{g_1}(s,X_{(0)})) = (m(x)\zeta,-\frac{1}{2}\nabla_x m(x)\zeta \cdot\zeta),
\end{equation}
where $\nabla_x m(x)\zeta\cdot\zeta$ is a vector whose components are $\frac{\partial m}{\partial x^j}\zeta\cdot\zeta$, $j=1,\dots, n$. Note that $m$ is independent of the time variable, we view $\nabla_x m$ and the full gradient $\nabla m=(0, \nabla_x m)$ as the same. 
Inserting the expression \eqref{Hfield} into \eqref{identity2} and comparing the last $(1+n)$ components gives
\begin{equation} \label{identity3}
\int^\infty_{-\infty} (\nabla_x m\zeta\cdot\zeta - 2B_{21}m\zeta + B_{22}\nabla_x m \zeta\cdot\zeta) \,ds = 0.
\end{equation}
Here the integration in $s$ can be extended to $\pm\infty$ due to the fact that $m$ is compactly supported. 

Notice that $g^{-1}_j$, $j=1,2$, are of the form \eqref{special form of g}, therefore $m=g^{-1}_1-g^{-1}_2$ is of the form 
\begin{equation} \label{special form of m}
\left(
\begin{array}{ccc}
   \ast & 0 & \ast_{1\times (n-1)}\\[.5em]
	0 & 0 & 0_{1\times (n-1)} \\[.5em]
	\ast_{(n-1)\times 1} & 0_{(n-1)\times 1} & \ast_{(n-1)\times (n-1)}
\end{array}
\right).
\end{equation}
In other words,
$$m=m_\lambda(x) dt^2+2 m_\omega(x) dt\otimes dx+m_h(x) dx^2,$$
where $m_\omega$ and $m_h$ have the form
\begin{equation}\label{special form of m 2}
( 0, \,\,\ast_{1\times (n-1)}) \quad \text{and}  \quad \left(
\begin{array}{cc}
	0 & 0_{1\times (n-1)} \\
	0_{(n-1)\times 1} & \ast_{(n-1)\times (n-1)}
\end{array}
\right)
\end{equation}
respectively.

\section{Proof of Theorem \ref{main theorem}}\label{proof of main thm}

Now we fix a future-pointing time-like covector $\zeta^{(0)}=(\varrho,\xi)$, $\varrho<-1$, $\xi\in \mathbb S^{n-1}$, which is close to light-like (see the remark at the end of Section 2.1), i.e. $\varrho$ is close to $-1$. Denote by $\xi^\perp:=\{y\in \mathbb R^n : y\cdot \xi=0\}$. Here $|\cdot|$ and `$\cdot$' are Euclidean norm and inner product. Let $z_{(0)}=(0,y)-\rho(-\varrho,\xi)$ for some $y\in \xi^\perp$ and $\rho>0$ large enough, so that $y-\rho\xi\notin \Omega$. Notice that $\rho$ can be chosen uniformly for all $y\in \xi^\perp$. In other words, $\overline \Omega\subset B_\rho$, where $B_\rho:=\{x\in\mathbb R^n: |x|<\rho\}$.


Next we apply the Fourier transform to \eqref{identity3} with respect to the spatial variable $y\in \xi^\perp$
\begin{equation*}
\int_{\xi^\perp} \int_{-\infty}^{\infty} e^{-i\eta\cdot y}\bigg(\nabla_x m(x) \zeta\cdot \zeta-2B_{21}m(x)\zeta+B_{22}\nabla_x m(x)\zeta\cdot \zeta\bigg)\,ds dy=0,
\end{equation*}
where $\eta\in \xi^\perp$ is the dual variable. Recall that $z=z_{g_1}(s,z_{(0)},\zeta^{(0)})=z_{g_1}(s,y,\varrho,\xi)$, $\zeta=\zeta_{g_1}(s,z_{(0)},\zeta^{(0)})=\zeta_{g_1}(s,y,\varrho,\xi)$, and $x=x_{g_1}$ is the spatial component of $z$.

Since $\eta\cdot\xi=0$, we deduce $\eta\cdot y = \eta \cdot (y+(s-\rho)\xi):=\eta\cdot x_\delta$ where $x_\delta:=x_\delta(s,y,\varrho,\xi)=y+(s-\rho)\xi$ is the spatial component of the geodesic with respect to the Minkowski metric $\delta$. Making the change of variable by $x_\delta\mapsto x$ we have
\begin{equation} \label{identity4}
\int_{\mathbb{R}^{n}} e^{-i\varphi(x,\eta)}(\nabla_x m\zeta\cdot\zeta- 2B_{21}m\zeta + B_{22}\nabla_x m \zeta\cdot\zeta) J^{-1}_1\;dx=0, \quad\quad \eta\in\xi^\perp.
\end{equation}
Here $\varphi(x,\eta):= \eta\cdot x_\delta(x)$ where $x_\delta(x)$ is the inverse of the change of variable $x_\delta\mapsto x$; $J_1$ is the Jacobian of the change of variable. For the following analysis, we can view the Jacobian the same as an identity, up to an $O(\epsilon)$ term in $C^{k-4}$, due to the closeness of $g_1$ to $\delta$.

Let us pause here to make some preparations for the proof later. First, we introduce the symbol class to be used. 
\begin{Definition}
Define
\begin{align*}
S^m_k :=& \left\{a(x,y,\xi)\in C^k(B_\rho\times B_\rho \times \mathbb{R}^{n}\backslash\{0\}): \text{ there exists a constant } C>0 \right. \\
  &  \text{ such that } |\partial^\alpha_x \partial^\beta_y \partial^\gamma_\xi a(x,y,\xi)| \leq C |\xi|^{m-\gamma} \text{ for } (x,y,\xi)\in B_\rho\times B_\rho \times \mathbb{R}^{1+n}\backslash\{0\} \\
  & \left. \text{ and } |\alpha|+|\beta|+|\gamma|\leq k  \right\}.
\end{align*} 
This space is equipped with the norm
$$\|a\|_{S^m_k} := \displaystyle\sum_{|\alpha|+|\beta|+|\gamma|\leq k} \sup_{(x,y,\xi)\in
B_\rho\times B_\rho \times \mathbb{R}^{n}\backslash\{0\}} \frac{|\partial^\alpha_x \partial^\beta_y \partial^\gamma_\xi a(x,y,\xi)|}{|\xi|^{m-\gamma}}$$
so that $a=O(\epsilon)$ in $S^m_k$ means $\|a\|_{S^m_k}=O(\epsilon)$.
\end{Definition}
This is the analog of H\"{o}rmander's symbol class but with finite regularity. This type of symbol class of finite regularity has been introduced before in \cite{SU97, SU98b} and some useful estimates have been developed therein.

We denote $\eta=(\eta_1,\eta')$ with $\eta'\in \mathbb R^{n-1}$. 
Let $\chi: \mathbb R^{n}\setminus \{0\}\to [0,\infty) $ be a smooth homogeneous function of order $0$ satisfying  
\begin{equation*}
\chi(\eta)=\left\{ \begin{array}{lr}
 0, \quad |\eta_1|/|\eta|< \mu/2,\\[.5em]
   1, \quad |\eta_1|/|\eta|> \mu,
          \end{array} \right.
\end{equation*}
for some small $\mu>0$. When $\eta\in \supp \chi$ (so $\eta_1\neq 0$), and $(\xi_1,\xi')=\xi\perp \eta$ (i.e. $\eta_1\xi_1+\eta'\cdot \xi'=0$), since $\xi\in \mathbb S^{n-1}$, then $\xi'\neq 0$. Thus we may fix some $p\in \mathbb S^{n-2}$, let 
\begin{equation}\label{choice of xi}
\xi_1(\eta,p)=\frac{-\eta'\cdot p}{\sqrt{|\eta'\cdot p|^2+|\eta_1 |^2}},\quad \xi'(\eta,p)=\frac{\eta_1\;p}{\sqrt{|\eta'\cdot p|^2+|\eta_1|^2}}.
\end{equation}
Our definition of $\xi(\eta,p)$ works in any dimensions, while the one in \cite{SU98} is defined specifically for 3D, due to the use of cross product.
Then we multiply \eqref{identity4} by the cut-off function $\chi$ to get
\begin{equation} \label{intidentity5}
\chi(\eta)\int_{\mathbb{R}^{n}} e^{-i\varphi(x,\eta)}(\nabla_x m\zeta\cdot\zeta - 2B_{21}m\zeta + B_{22}\nabla_x m \zeta\cdot\zeta) J^{-1}_1\;dx=0.
\end{equation}

Next we compute the integrand in \eqref{intidentity5}. The closeness condition \eqref{xiclose} implies that for $x$ and $\eta$ in a compact set
\begin{equation} \label{phaseclose}
\varphi(x,\eta) = x\cdot\eta + O(\epsilon) \quad \text{in } S^1_{k-3}.
\end{equation}
The closeness condition \eqref{xiclose} also ensures that for $x$ and $\eta$ in a compact set, $\zeta(\varrho,\eta) = \zeta^{(0)}(\varrho,\eta) + O(\epsilon)$ in $S^0_{k-3}$.  
Then the components of the term $\nabla_x m \zeta\cdot\zeta$ in \eqref{intidentity5} can be computed as
$$\frac{\partial m}{\partial x_j}\zeta\cdot\zeta =\frac{\partial m_{\lambda}}{\partial x_j}\varrho^2 + 2\frac{\partial m_{\omega}}{\partial x_j}\xi\varrho+ \frac{\partial m_{h}}{\partial x_j}\xi\cdot\xi  + O(\epsilon)\frac{\p m}{\p x_j}$$
where $O(\epsilon)$ is in $S^0_{k-3}$. Note that $\varrho$ is constant along the geodesics of $g_1$. The other terms in the integrand of \eqref{intidentity5} can be computed analogously. Recall that $B_{21},\; B_{22}=O(\epsilon)$ in $S^0_{k-4}$, and $J^{-1}_1=1+O(\epsilon)$ in $S^0_{k-4}$ too.
With these expressions \eqref{intidentity5} becomes
\begin{align*}
\chi(\eta) \int_{\mathbb{R}^{n}} & e^{-i\varphi(x,\eta)} ( \nabla_{x} m_{\lambda} \varrho^2 + 2 \nabla_{x} m_{\omega} \xi\varrho+ \nabla_x m_{h}\xi\cdot\xi  + \\ 
 & \sum^n_{\alpha,\beta=0} D_{\alpha\beta} m_{\alpha\beta} + \sum^n_{l=1}\sum^n_{\alpha,\beta=0} E_{\alpha\beta l} \frac{\partial m_{\alpha\beta}}{\partial x_l}) \;dx=0
\end{align*}
with 
\begin{equation} \label{CDsmall}
D_{\alpha\beta}=D_{\alpha\beta}(x,p,\varrho,\eta)=O(\epsilon),\; E_{\alpha\beta l}=E_{\alpha\beta l}(x,p,\varrho,\eta)=O(\epsilon) \quad \text{ in } S^0_{k-4} \text{ for } \eta \in \supp \chi. 
\end{equation}
In other words,
\begin{align}
 & \chi(\eta) \int_{\mathbb{R}^{n}} e^{-i\varphi(x,\eta)} \left( \nabla_{x} m_{\lambda} \varrho^2 + 2 \nabla_{x} m_{\omega} \xi\varrho+ \nabla_x m_{h}\xi\cdot\xi   \right) \;dx \nonumber \\
= & - \chi(\eta) \int_{\mathbb{R}^{n}} e^{-i\varphi(x,\eta)} \left( \sum^n_{\alpha,\beta=0} D_{\alpha\beta} m_{\alpha\beta} + \sum^n_{l=1}\sum^n_{\alpha,\beta=0} E_{\alpha\beta l} \frac{\partial m_{\alpha\beta}}{\partial x_l}
\right) \;dx.  \label{intidentity6}
\end{align}
The right-hand side is an oscillatory integral with the phase function $\varphi$ satisfying \eqref{phaseclose} and the amplitude of magnitude $O(\epsilon)$ in $S^0_{k-4}$. Notice that these integrands are all supported in $\Omega$.

Taking use of the structure of $m$, see \eqref{special form of m} and \eqref{special form of m 2}, we introduce the notations
\begin{align*}
A_1 m (\eta) & :=\chi(\eta) \int e^{-i\varphi(x,\eta)} \nabla_x m_\lambda (x) \,dx,\\
A_2 m (\eta, p) & := \chi(\eta)\int e^{-i\varphi(x,\eta)} \psi_p(\eta) \nabla_x m_\omega (x) p \, dx,\\
A_3 m (\eta, p) & :=\chi(\eta) \int e^{-i\varphi(x,\eta)} \psi^2_p(\eta) \nabla_x m_h (x) p \cdot p \, dx.
\end{align*}
Here $\psi_p(\eta):=\frac{\eta_1}{\sqrt{|\eta'\cdot p|^2+|\eta_1|^2}}$, which is a homogeneous function of order $0$ in the support of $\chi$. 
Thus we can write \eqref{intidentity6} as
\begin{align*}
Am(\varrho, \eta,p):=\varrho^2 A_1 m (\eta)+ & 2 \varrho A_2 m (\eta,p)+A_3 m(\eta,p) \\
& =\chi(\eta)\int e^{-i\varphi(x,\eta)}(O(\epsilon) m(x)+O(\epsilon) \nabla m(x))\, dx.
\end{align*}
Notice that $A_2 m$ is odd in $p$, while $A_3 m$ is even in $p$ for $p\in \mathbb S^{n-2}$. Therefore
\begin{align*}
A_2 m (\eta,p) & =\frac{1}{4\varrho} \left( A\, m(\varrho,\eta,p)-A\, m(\varrho,\eta,-p) \right),\\
\varrho^2 A_1 m (\eta)+A_3 m (\eta,p) & =\frac{1}{2} \left( A\, m(\varrho,\eta,p)+A\, m(\varrho,\eta,-p) \right).
\end{align*}
Moreover, let $-1>\varrho_1>\varrho_2$ be close to $-1$ and fixed (so the choice of $\epsilon$ later is independent of $\varrho$), since $A_3 m_h$ is independent of $\varrho$, we get that
\begin{align*}
& A_1 m (\eta)=\frac{1}{2(\varrho_2^2-\varrho_1^2)}\left( A\,m(\varrho_2,\eta,p)+A\,m(\varrho_2,\eta,-p)-A\,m(\varrho_1,\eta,p)-A\,m(\varrho_1,\eta,-p)\right),\\
& A_3 m(\eta,p)=\frac{1}{2} \left( A\, m(\varrho,\eta, p)+A\, m(\varrho,\eta, -p) \right)-\varrho^2 A_1 m (\eta).
\end{align*}
In other words, one can determine $A_j m$, $j=1,2,3$, and
\begin{equation}\label{A_j m}
A_j m=\chi(\eta)\int e^{-i\varphi(x,\eta)}(O(\epsilon) m(x)+O(\epsilon) \nabla m(x)) \, dx
\end{equation}
with $O(\epsilon)$ in $S^0_{k-4}$.

Let $a(x)$ be a smooth cut-off function with $\supp{a}\subset B_\rho$ and $a(x)=1$ for $x\in\Omega$, then $m=a(x)m$, $\nabla m=a(x)\nabla m$.
We multiply the identities \eqref{A_j m} by $a(y)e^{i\varphi(y,\eta)}\psi_p^{j-1}(\eta)$ and integrate in $\eta$ to obtain 
\begin{equation}\label{A_j m 2}
\begin{split}
P_j m(y,p) & :=a(y)\int e^{i\varphi(y,\eta)}\psi_p^{j-1}(\eta)A_j m(\eta,p) \,d\eta\\
& =a(y)\int \int \chi(\eta) e^{i(\varphi(y,\eta)-\varphi(x,\eta))} \psi_p^{j-1}(\eta) a(x) (O(\epsilon) m(x)+O(\epsilon) \nabla m(x))\, dx d\eta.
\end{split}
\end{equation}
In particular
\begin{equation}\label{P_j m}
\begin{split} 
P_1 m(y,p) &= a(y) \int \int
\chi(\eta) e^{i(\varphi(y,\eta) - \varphi(x,\eta))}a(x) \nabla m_\lambda (x) \;dxd\eta, \\
P_2 m(y,p) &= a(y) \int \int
\chi(\eta) e^{i(\varphi(y,\eta) - \varphi(x,\eta))} \psi_p^2(\eta) a(x) \nabla m_\omega (x) p \;dxd\eta, \\
P_3 m(y,p) &= a(y) \int \int
\chi(\eta) e^{i(\varphi(y,\eta) - \varphi(x,\eta))} \psi_p^4(\eta) a(x) \nabla m_h (x) p\cdot p \;dxd\eta.
\end{split}
\end{equation}

Next we would like to make a change of variable so that the phase function is simplified. By Taylor's expansion
$$\varphi(y,\eta) - \varphi(x,\eta) = (y-x)\cdot \theta(x,y,\eta), \quad\quad \theta(x,y,\eta):= -\int^1_0 \nabla\varphi(y+t(x-y),\eta) \,dt.$$
It is easy to see that $\theta$, as a function of $\eta$, is homogeneous of order $1$ and $\theta = \eta + O(\epsilon)$ in $S^1_{k-4}$. Therefore, the map $\eta\mapsto \theta$ is a change of variable whose Jacobian $J_2:=\det(\frac{d\theta}{d\eta})$ satisfies $J_2 = 1 + O(\epsilon)$ in $S^0_{k-5}$. Replacing $\eta\mapsto \theta$ in \eqref{A_j m 2} and \eqref{P_j m} gives
\begin{equation}\label{eta to theta 1}
\begin{split} 
P_1 m(y,p) &= \int \int e^{i(y-x)\cdot\theta} a(y)\chi (\eta(x,y,\theta)) a(x) \nabla m_\lambda (x) J^{-1}_2\;dxd\theta, \\
P_2 m(y,p) &= \int \int e^{i(y-x)\cdot\theta} a(y)\chi (\eta(x,y,\theta)) \psi_p^2(\eta(x,y,\theta)) a(x) \nabla m_\omega (x) p J^{-1}_2\;dxd\theta, \\
P_3 m(y,p) &= \int \int e^{i(y-x)\cdot\theta} a(y)\chi (\eta(x,y,\theta)) \psi_p^4(\eta(x,y,\theta)) a(x) \nabla m_h (x) p\cdot p  J^{-1}_2\;dxd\theta,
\end{split}
\end{equation}
and for each $j=1,2,3$
\begin{equation}\label{eta to theta 2}
\begin{split}
P_j m(y,p) = & \\
\int \int e^{i(y-x)\cdot\theta} & a(y)\chi (\eta(x,y,\theta)) \psi_p^{j-1}(\eta(x,y,\theta)) a(x) (O(\epsilon) m(x)+O(\epsilon) \nabla m(x)) J^{-1}_2\, dx d\theta.
\end{split}
\end{equation}
In view of $ J^{-1}_2 = 1+O(\epsilon)$ in $S^0_{k-5}$ and $\eta=\theta+O(\epsilon)$ in $S^1_{k-4}$, we get for any $l\geq 0$ 
$$a(y) \chi(\eta(x,y,\theta)) \psi_p^l (\eta(x,y,\theta)) a(x)  J^{-1}_2 = a(y)\chi(\theta) \psi_p^l(\theta) a(x) + O(\epsilon) \quad\quad \text{ in } S^0_{k-5}.$$
Combine \eqref{eta to theta 1} and \eqref{eta to theta 2}, the above analysis implies
\begin{equation}\label{eta to theta 3}
\begin{split}
\int \int e^{i(y-x)\cdot\theta} a(y)\chi (\theta)  & a(x) \nabla m_\lambda (x) \;dxd\theta\\ 
& =\int \int e^{i(y-x)\cdot\theta} (O(\epsilon) m(x)+O(\epsilon) \nabla m(x))\, dx d\theta,\\
\int \int e^{i(y-x)\cdot\theta} a(y)\chi (\theta) \psi^2_p(\theta) & a(x) \nabla m_\omega (x) p  \;dxd\theta\\
 & =\int \int e^{i(y-x)\cdot\theta} (O(\epsilon) m(x)+O(\epsilon) \nabla m(x))\, dx d\theta,\\
\int \int e^{i(y-x)\cdot\theta} a(y)\chi (\theta) \psi^4_p(\theta) & a(x) \nabla m_h (x) p\cdot p \;dxd\theta\\
 & =\int \int e^{i(y-x)\cdot\theta} (O(\epsilon) m(x)+O(\epsilon) \nabla m(x))\, dx d\theta.
\end{split}
\end{equation}
The $O(\epsilon)$ on the right hand side of the above equalities are in $S^0_{k-5}$.

To simplify the notations, we still use $P_j m$, $j=1,2,3$, to denote the left hand side of the three equalities in \eqref{eta to theta 3} respectively.
The above argument implies that
\begin{Lemma}\label{P_j new}
For $j=1,2,3$, 
$$P_j m(y,p)=\int \int e^{i(y-x)\cdot\theta} O(\epsilon)m(x) \;dxd\theta +\int \int e^{i(y-x)\cdot\theta} O(\epsilon)\nabla_x m(x) \;dxd\theta,$$
where $O(\epsilon)$ are in $S^0_{k-5}$.
\end{Lemma}

To deal with the integrals with amplitudes of order $O(\epsilon)$, we recall the following result established in \cite[Theorem A1]{SU97}. 
\begin{Lemma}\label{FIO}
Let $P$ be the operator
$$Pf(x):=\int \int e^{i(y-x)\cdot\xi}a(x,y,\xi) f(y) \,dyd\xi$$
If for all $\xi\in\mathbb{R}^n$,
$$\sum_{|\alpha|+|\beta|\leq 2n+1} \int \int  |\partial^\alpha_x\partial^\beta_y a(x,y,\xi)| \,dxdy \leq M.$$
Then $P:L^2(\mathbb{R}^n)\rightarrow L^2(\mathbb{R}^n)$ is a bounded operator, and 
$\|P\|_{L^2\to L^2} \leq CM$ for some constant $C>0$.
\end{Lemma}

Applying Lemma \ref{FIO} to Lemma \ref{P_j new}, we have for $k-5\geq 2n+1$ (this is the reason we assume that $g_j$ is close to $\delta$ in $C^{k}$-norm, $k\geq 2n+6$,  in Theorem \ref{main theorem}),
$$
P_j m(y,p)= O(\epsilon\|m\|) \; + \; O(\epsilon\|\nabla m\|) \text{ in } L^2(\mathbb{R}_y^{n}),\quad j=1,2,3,
$$
where $\|m\|:=\|m\|_{L^2(\mathbb R^n)}$. By the Poincar\'{e}'s inequality $\|m\|_{L^2} \leq C \|\nabla m\|_{L^2}$ (note that $m$ is compactly supported), the first term on the right hand side can be absorbed into the second term, thus
$$
P_j m(y,p)= O(\epsilon\|\nabla m\|)  \quad \text{ in } L^2(\mathbb{R}^{n}).
$$
It follows that 
\begin{align*}
\|\chi(\theta)\theta \hat m_\lambda(\theta)\|_{L^2}^2& =(\nabla m_\lambda, P_1 m )_{L^2(\mathbb R^n_y)}=O(\epsilon \|\nabla m\|^2),\\
\|\chi(\theta)\psi_p(\theta)\theta \hat m_\omega (\theta) p\|_{L^2}^2& =(\nabla m_\omega p, P_2 m )_{L^2(\mathbb R^n_y)}=O(\epsilon \|\nabla m\|^2),\\
\|\chi(\theta)\psi^2_p(\theta) \theta \hat m_h (\theta) p\cdot p\|_{L^2}^2& =( \nabla m_h p\cdot p, P_3 m )_{L^2(\mathbb R^n_y)}=O(\epsilon \|\nabla m\|^2),
\end{align*}
which 
yield
\begin{align*} \label{lc}
\theta \hat m_\lambda (\theta) & =O(\sqrt{\epsilon}\|\nabla m\|),\\
\psi_p(\theta) \theta \hat m_\omega(\theta) p & =O(\sqrt{\epsilon}\|\nabla m\|),\\
\psi^2_p(\theta)\theta \hat m_h (\theta) p\cdot p & =O(\sqrt{\epsilon}\|\nabla m\|)
\end{align*}
in $L^2(|\theta_1|/|\theta|>\mu)$ with $\mu>0$ a small parameter. Here $\hat m$ is the Fourier transform of $m$.

Recall that $\psi_p(\theta)=\frac{\theta_1}{\sqrt{|\theta'\cdot p|^2+|\theta_1|^2}}$, therefore when $|\theta_1|/|\theta|>\mu$, we have $|\psi_p(\theta)|>\mu$. This implies
$$\theta \hat m_\lambda (\theta) =O(\sqrt{\epsilon}\|\nabla m\|),\quad  \theta \hat m_\omega(\theta) p  =O(\frac{\sqrt{\epsilon}}{\mu}\|\nabla m\|),\quad \theta \hat m_h (\theta) p\cdot p  =O(\frac{\sqrt{\epsilon}}{\mu^2}\|\nabla m\|)$$
in $L^2(|\theta_1|/|\theta|>\mu)$.
Notice that $p$ is an arbitrary vector on $\mathbb S^{n-2}$, taking into account the structure of $m_\omega$ and $m_h$ \eqref{special form of m 2}, we derive that
\begin{equation}\label{>mu}
\theta \hat m_\lambda (\theta) =O(\sqrt{\epsilon}\|\nabla m\|),\quad  \theta \hat m_\omega(\theta)   =O(\frac{\sqrt{\epsilon}}{\mu}\|\nabla m\|),\quad \theta \hat m_h (\theta)  =O(\frac{\sqrt{\epsilon}}{\mu^2}\|\nabla m\|)
\end{equation}
in $L^2(|\theta_1|/|\theta|>\mu)$. 

Then we estimate the $L^2$-norm of $\theta \hat m(\theta)$ in the region $\{|\theta_1|/|\theta|\leq \mu\}$, which is unbounded.

\begin{Lemma}\label{fourier estimate in cone}
Assume that $\|m\|_{H^2}\leq K\|m\|_{H^1}$ for some constant $K>0$, then 
\begin{equation}\label{L2 norm of m in cone}
\|\theta \hat m(\theta)\|_{L^2(|\theta_1|/|\theta|\leq \mu)}\leq \frac{1}{3}\|\nabla m\|_{L^2}
\end{equation}
for sufficiently small $\mu$, depending on $K$.
\end{Lemma}
\begin{proof}
\begin{align*}
\|\theta \hat m(\theta)\|^2_{L^2(|\theta_1|/|\theta|\leq \mu)}=\int_{|\theta_1|/|\theta|\leq \mu,\; |\theta|<r} |\theta \hat m(\theta)|^2\, d\theta+ \int_{|\theta_1|/|\theta|\leq \mu,\; |\theta|\geq r} |\theta \hat m(\theta)|^2\, d\theta.
\end{align*}
We need to estimate the two terms on the right hand side. Denote the bounded region $\{|\theta_1|/|\theta|\leq \mu,\; |\theta|<r\}$ by $U_{\mu, r}$, then
\begin{align*}
\int_{U_{\mu, r}} |\theta \hat m(\theta)|^2\, d\theta =\int_{U_{\mu,r}} |\int_{\Omega} e^{-i\theta x} \nabla m(x)\, dx|^2\, d\theta & \leq \int_{U_{\mu,r}} (\int_{\Omega} 1\, dx \int |\nabla m(x)|^2\, dx)\, d\theta\\
& =\mbox{Vol} (U_{\mu, r}) \mbox{Vol} (\Omega) \|\nabla m\|^2_{L^2},
\end{align*}
where Vol$(U_{\mu,r})=f^2(\mu) r^n$ with a positive function $f(\mu)\to 0$ as $\mu\to 0$, independent of $r$. Here Vol$(U)$ means the volume of a set $U$.
\begin{align*}
\int_{|\theta_1|/|\theta|\leq \mu,\; |\theta|\geq r} |\theta \hat m(\theta)|^2\, d\theta & \leq \int_{|\theta|\geq r} |\theta \hat m(\theta)|^2\, d\theta=\int_{|\theta|\geq r} \frac{|\theta|^2}{|\theta|^2} |\theta\hat m(\theta)|^2\, d\theta \\
& \leq r^{-2} \int_{|\theta|\geq r} (|\theta|^2 |\hat m(\theta)|)^2\, d\theta \leq C r^{-2} \|m\|^2_{H^2}.
\end{align*}
Apply the assumption $\|m\|_{H^2}\leq K\|m\|_{H^1}$, by the Poincar\'e inequality we have
$$\int_{|\theta_1|/|\theta|\leq \mu,\; |\theta|\geq r} |\theta \hat m(\theta)|^2\, d\theta\leq C r^{-2}\|m\|^2_{H^1}\leq C r^{-2} \|\nabla m\|^2_{L^2}$$
Combining the above argument
$$\|\theta \hat m(\theta)\|_{L^2(|\theta_1|/|\theta|\leq \mu)}\leq C(f(\mu) r^{n/2}+r^{-1}) \|\nabla m\|_{L^2}.$$
In particular, let $r=4C$, recall that $f(\mu)\to 0$ as $\mu\to 0$,
\begin{align*}
\|\theta \hat m(\theta)\|_{L^2(|\theta_1|/|\theta|\leq \mu)}\leq (C f(\mu)+\frac{1}{4}) \|\nabla m\|_{L^2}\leq \frac{1}{3}\|\nabla m\|_{L^2}
\end{align*}
when $\mu$ is sufficiently small.
\end{proof}

Now by the a priori estimate in Theorem \ref{main theorem}, $\|g_1-g_2\|_{H^2}\leq K\|g_1-g_2\|_{H^1}$, the estimate holds after diffeomorphisms that are sufficiently close to the identity map. Therefore we have $m=\tilde g_1^{-1}-\tilde g_2^{-1}$ satisfies similar estimate with a new fixed bound $K$. Apply Lemma \ref{fourier estimate in cone}, we get the estimate \eqref{L2 norm of m in cone} for sufficiently small $\mu>0$.

Finally, let $\mu=\epsilon^{1/8}$ for $\epsilon>0$ sufficiently small, we combine \eqref{>mu} and \eqref{L2 norm of m in cone} to get
\begin{align*}
\|\nabla m\|_{L^2}=\|\theta \hat m(\theta)\|_{L^2(\theta)} & \leq \|\theta \hat m(\theta)\|_{L^2(|\theta_1|/|\theta|>\mu)}+ \|\theta \hat m(\theta)\|_{L^2(|\theta_1|/|\theta|\leq \mu)}\\
& \leq (\frac{1}{3}+\frac{1}{3})\|\nabla m\|_{L^2} = \frac{2}{3} \|\nabla m\|_{L^2},
\end{align*}
so $\nabla m=0$. Note that $m$ is compactly supported, we get immediately that $m\equiv 0$. This proves Theorem \ref{main theorem}.

\section{Sketch of the proof of Theorem \ref{Riemannian case}}

In this section, we give a sketch of the proof of the improved result for the Riemannian case of Theorem \ref{Riemannian case}. 

Let $h_1$ and $h_2$ be two Riemannian metrics in $\mathbb{R}^{n}$, $n\geq 2$, which differ from the Euclidean metric $e$ only on $\overline{\Omega}$. Suppose $\|h_j-e\|_{C^{k}_0(\overline{\Omega})}\leq \epsilon$, $j=1,2$, $k\geq 2n+6$, for some sufficiently small $\epsilon>0$, and $h_1$, $h_2$ have the same boundary distance function, i.e. 
$$d_{h_1}|_{\p\Omega \times \p\Omega} = d_{h_2}|_{\p\Omega \times \p\Omega}.$$
By applying the argument of Lemma \ref{srelation}, see also \cite[Lemma 2.1]{SU98}, $h_1$ and $h_2$ have the same scattering relation. Then following the argument of Section \ref{construct diffeomorphism}, in particular let the hyperplane $H$ be $\{x^1=-\rho\}\subset \mathbb R^n\setminus \Omega$, we obtain diffeomorphisms $\psi_j$, $j=1,2$, so that $\tilde h_j:=\psi_j^* h_j$ has the form 
\begin{equation*} 
\left(
\begin{array}{cc}
	1 & 0_{1\times (n-1)} \\
	0_{(n-1)\times 1} & \ast_{(n-1)\times (n-1)}
\end{array}
\right).
\end{equation*}
Therefore $m=\tilde h_1^{-1}-\tilde h_2^{-1}$ takes the form
\begin{equation}\label{special form of m riemannian}
\left(
\begin{array}{cc}
	0 & 0_{1\times (n-1)} \\
	0_{(n-1)\times 1} & \ast_{(n-1)\times (n-1)}
\end{array}
\right).
\end{equation}

Similar to \eqref{identity3}, the integral identity w.r.t. the Riemannian geodesic flows implies
\begin{equation}\label{riem integral 1}
\int (\nabla m \xi\cdot\xi-2B_{21} m\xi +B_{22} \nabla m \xi\cdot\xi)\, ds=0,
\end{equation}
with $B_{21}, B_{22}=O(\epsilon)$ in $S^0_{k-4}$.
Then we follow the argument in Section \ref{proof of main thm}, denote $\eta=(\eta_1,\eta'),\; \eta'\in\mathbb R^{n-1}$ and $p\in \mathbb S^{n-2}$, let 
\begin{equation*}
\xi_1(\eta,p)=\frac{-\eta'\cdot p}{\sqrt{|\eta'\cdot p|^2+|\eta_1 |^2}},\quad \xi'(\eta,p)=\frac{\eta_1\;p}{\sqrt{|\eta'\cdot p|^2+|\eta_1|^2}}
\end{equation*}
give the initial direction $\xi^{(0)}=\xi(\eta, p)$. 
However, we modify the homogeneous cut-off function $\chi\in C^\infty(\mathbb R^n\setminus \{0\})$ as
\begin{equation*}
\chi_p(\eta)=\left\{ \begin{array}{lr}
 0, \quad \frac{|\eta'\cdot p|+|\eta_1|}{|\eta|}< \mu/2,\\[.5em]
   1, \quad \frac{|\eta'\cdot p|+|\eta_1|}{|\eta|}> \mu,
          \end{array} \right.
\end{equation*}
for some small $\mu>0$, so it avoids the singularity of $\xi(\eta,p)$, i.e. $\eta_1=0,\; \eta'\perp p$. Notice that if $n=2$, then $\chi_p(\eta)\equiv 1$ for $\eta\in \mathbb R^2\setminus \{0\}$.

Recall that the diffeomorphism $\psi_j$, $j=1,2$, maps straight lines to geodesics of metric $g_j$ with initial velocity $\xi^{(0)}=\pm e_1=\pm\p_{x^1}$ (so $\xi'(\eta,p)=0$, i.e. $\eta_1=0$), thus the pullback of $\xi=\xi(s)$ is constantly $\pm e_1$. This implies 
$$\xi=\xi(\eta,p)+\frac{\eta_1}{|\eta|}O(\epsilon)=\frac{\eta_1}{\sqrt{|\eta'\cdot p|^2+|\eta_1|^2}}\left[(\frac{-\eta'\cdot p}{\eta_1}, p)+O(\epsilon)\right]$$
with $O(\epsilon)\in S^0_{k-4}$. Similarly, in view of the definition of $B_{ij}$ in \eqref{derivative}, when $\xi^{(0)}=\pm e_1$, one can check that $B_{21}\equiv 0$. Thus generally $B_{21}=\frac{\eta_1}{|\eta|} \tilde B_{21}$ with $\tilde B_{21}=O(\epsilon)\in S^0_{k-5}$. We remark that the above asymptotic expressions does not hold for the Lorentzian case in Section \ref{proof of main thm}, since the pullback of timelike geodesics by a time invariant diffeomorphism are not straight lines in general. In particular, taking into account the structure of $m$ \eqref{special form of m riemannian}, the above analysis implies
\begin{align*}
\nabla m \xi\cdot \xi-2B_{21} m \xi+B_{22} \nabla m \xi\cdot \xi=\frac{\eta_1^2}{|\eta'\cdot p|^2+|\eta_1|^2} \bigg(\nabla m\; p\cdot p + D\; m+E\; \nabla m\bigg),
\end{align*}
where $D=O(\epsilon), E=O(\epsilon)$ in $S^0_{k-5}$, are matrix functions. Thus in the support of $\chi_p$, we can cancel the common factor $\eta_1^2/(|\eta'\cdot p|^2+\eta_1^2)$ from \eqref{riem integral 1} to get
$$\int (\nabla m\; p\cdot p + D\; m+E\; \nabla m)\, ds=0.$$

\noindent {\bf Remark:} {\it In the support of $\chi_p$, the common factor $\eta_1^2/(|\eta'\cdot p|^2+\eta_1^2)$, which is exactly $\psi_p^2(\eta)$ in Section \ref{proof of main thm}, could be zero, then its inverse will blow up. This is the reason we choose a different cut-off function $\chi$ in the Lorentzian case.}

Next we follow the approach in Section \ref{proof of main thm} by taking Fourier transform of the above integral identity, and to get eventually 
\begin{equation}
\begin{split}
\int \int e^{i(y-x)\cdot\theta} a(y)\chi_p (\theta) & a(x) \nabla m (x) p\cdot p \;dxd\theta\\
 & =\int \int e^{i(y-x)\cdot\theta} (O(\epsilon) m(x)+O(\epsilon) \nabla m(x))\, dx d\theta.
\end{split}
\end{equation}
Here $a(x)$ is a smooth cut-off function with $\supp{a}\subset B_\rho$ and $a(x)=1$ for $x\in\Omega$. The $O(\epsilon)$ on the right hand side of the above equality are in $S^0_{k-5}$.

It follows that 
\begin{align*}
\|\chi_p(\theta) \theta \hat m (\theta) p\cdot p\|_{L^2}^2 =O(\epsilon \|\nabla m\|^2),
\end{align*}
which 
yields
\begin{align*} \label{lc}
\theta \hat m (\theta) p\cdot p & =O(\sqrt{\epsilon}\|\nabla m\|)
\end{align*}
in $L^2(\frac{|\theta'\cdot p|+|\theta_1|}{|\theta|}>\mu)$ with $\mu>0$ a small parameter. Here $\hat m$ is the Fourier transform of $m$.

When $n=2$, $\{\frac{|\theta'\cdot p|+|\theta_1|}{|\theta|}>\mu\}=\mathbb R^2\setminus \{0\}$ for $\mu$ sufficiently small. Then it follows immediately that $\theta \hat m(\theta)=O(\sqrt{\epsilon}\|\nabla m\|)$ in $L^2(\mathbb R^2)$.

When $n\geq 3$, the set $\{\frac{|\theta'\cdot p|+|\theta_1|}{|\theta|}\leq \mu\}$ is a conic neighborhood of the $n-2$ dimensional plane $\{\theta_1=0, \theta'\cdot p=0\}$, with the vertex at the origin.
Notice that $p$ is an arbitrary vector on $\mathbb S^{n-2}$, taking into account the structure of $m$ \eqref{special form of m riemannian}, 
we choose $(n-1)^2$ vectors $p\in \mathbb S^{n-2}$, similar to the choice in \cite{SU98}, so that for $\mu$ is sufficiently small, we can derive that
$$\theta \hat m_{ij} (\theta)=O(\sqrt{\epsilon}\|\nabla m\|)\quad \mbox{in}\quad L^2(\mathbb R^n), \quad i,j=2,\cdots, n.$$



Finally, let $\epsilon>0$ be sufficiently small, we obtain
\begin{align*}
\|\nabla m\|_{L^2}=\|\theta \hat m(\theta)\|_{L^2(\theta)} & \leq C\sqrt{\epsilon} \|\nabla m\|_{L^2} \leq \frac{1}{2} \|\nabla m\|_{L^2},
\end{align*}
so $\nabla m=0$. Since $m$ is compactly supported, we get immediately that $m\equiv 0$. This proves Theorem \ref{Riemannian case}.


\end{document}